\documentclass[smallextended,envcountsect]{svjour3}

\smartqed 
\usepackage{amsmath,amsfonts,amsthm} 
\usepackage{graphicx}
\usepackage{mathtools}
\usepackage{setspace}
\usepackage{colonequals}
\usepackage{tabulary}
\usepackage{booktabs}
\usepackage {tikz}
\usepackage[misc,geometry]{ifsym}
\usepackage{pgfkeys}
\usetikzlibrary{external}
\usepackage{pgfplots}
\newcolumntype{K}[1]{>{\centering\arraybackslash}p{#1}}
\usepackage{import} 
\usepackage{hyperref}

\usepackage{enumitem}
\setlist[itemize,1]{label=$\diamond$}
\setlist[itemize,2]{label=$-$}
\journalname{JOTA}

\begin{document}

\title{Exact Worst-Case Convergence Rates of the Proximal Gradient Method for Composite Convex Minimization 
}

\titlerunning{Exact worst-case convergence rates of the proximal gradient method}
\author{
  Adrien B. Taylor
  \and
  Julien M. Hendrickx
  \and
  Fran\c cois Glineur
}

\institute{Adrien B. Taylor (\Letter) \at
             Universit\'e catholique de Louvain, ICTEAM \\
              B-1348 Louvain-la-Neuve, Belgium\\
              adrienbtaylor@gmail.com 
           \and
           Julien M. Hendrickx \at
             Universit\'e catholique de Louvain, ICTEAM \\
              B-1348 Louvain-la-Neuve, Belgium\\
              julien.hendrickx@uclouvain.be
           \and
           Fran\c cois Glineur  \at
             Universit\'e catholique de Louvain, ICTEAM/CORE \\
              B-1348 Louvain-la-Neuve, Belgium\\
              francois.glineur@uclouvain.be
}

\newcommand{\pinf}{\infty}
\newcommand{\eqdef}{\overset{\text{def}}{=}}
\newcommand{\cinf}{\cup\left\{\pinf\right\}}
\newcommand{\R}{\mathbb{R}}
\newcommand{\Rinf}{\mathbb{R}\cinf}
\newcommand{\Rp}{\R^+}
\newcommand{\Rd}{\mathbb{R}^d}
\newcommand{\E}{\mathbb{E}}
\newcommand{\Es}{\mathbb{E}^*}

\newcommand{\norm}[1]{{\left\lVert#1\right\rVert}}
\newcommand{\normsq}[1]{{\left\lVert#1\right\rVert^2}}

\newcommand{\relint}[1]{\mathrm{relint}\left(#1\right)}
\newcommand{\inner}[2]{{\left< #1, #2\right>}}
\newcommand{\innerS}[2]{{\langle #1, #2 \rangle}}
\newcommand{\innerE}[2]{{\langle #1, #2 \rangle}_{\E}}
\newcommand{\innerEs}[2]{{\langle #1, #2 \rangle}_{\Es}}
\newcommand{\argmin}[1]{\underset{#1}{\text{argmin }}}
\newcommand{\argmax}[1]{\underset{#1}{\text{argmax }}}
\newcommand{\prox}[2]{{\bf p}_{#1}\left(#2\right)}
\newcommand{\goto}[1]{\overset{\mu\rightarrow 0}{\rightarrow}#1}
\newcommand{\fdefsimple}{\E\rightarrow\R}
\newcommand{\fdef}{\E\rightarrow\Rinf}
\newcommand{\fdefzero}{\E\rightarrow\left\{0,\pinf \right\}}
\newcommand{\fsdef}{\Es\rightarrow\Rinf}
\DeclarePairedDelimiter{\ceil}{\lceil}{\rceil}

\newcommand{\secref}[1]{Section~\ref{#1}}
\newcommand{\figref}[1]{\figurename~\ref{#1}}

\newcommand{\rev}[1]{#1}

\newcommand{\defeq}{\rev{\colonequals}}
\newcommand{\DO}{{$\norm{x_N-x_\star}$}}
\newcommand{\OFA}{{$F(x_N)-F_\star$}}
\newcommand{\RGN}{{$\norm{{\nabla} f(x_N)+s_N}$}}
\newcommand{\DOt}{{distance to optimality}}
\newcommand{\OFAt}{{objective function accuracy}}
\newcommand{\RGNt}{{residual gradient norm}}
\date{Received: date / Accepted: date}

\maketitle

\begin{abstract}
We study the worst-case convergence rates of the proximal gradient method for minimizing the sum of a smooth strongly convex function and a non-smooth convex function, whose proximal operator is available. 

We establish the exact worst-case convergence rates of the proximal gradient method in this setting for any step size and for different standard performance measures: \OFAt, \DOt\ and residual gradient norm. 

The proof methodology relies on recent developments in performance estimation of first-order methods, based on semidefinite programming. In the case of the proximal gradient method, this methodology allows obtaining exact and non-asymptotic worst-case guarantees, that are conceptually very simple, although apparently new. 

On the way, we discuss how strong convexity can be replaced by weaker assumptions, while preserving the corresponding convergence rates. We also establish that the same fixed step size policy is optimal for all three performance measures. Finally, we extend recent results on the worst-case behavior of gradient descent with exact line search to the proximal case.

\end{abstract}

\section{Introduction} \label{intro}

The proximal gradient method is a well-known extension to the standard gradient method for minimizing the sum of a smooth function with a non-smooth convex one. Numerous variants of this method were studied in the literature with a corresponding variety of results depending on the particular assumptions made on the optimization problems of interest.

In this work, we are concerned with the case where the smooth term is also strongly convex. In this situation, the proximal gradient method is known to converge linearly, as in the case of the gradient method (and its projected variant), whose convergence is already widely studied --- see among others~\cite[Section 1.4: Theorem 3]{Book:polyak1987},~\cite[Theorem 2.1.14]{Book:Nesterov},~\cite[Section 5.1]{primerryu} and~\cite[Section 4.4]{lessard2014analysis} (with slight variants in the assumptions: depending on whether the smooth function is required to be twice differentiable or not).

Our main contribution is to prove exact worst-case
rates of convergence of the method for the three standard optimality measures: distance to optimality, objective function accuracy and residual gradient norm. 
These tight rates were to the best of our knowledge only known for the first of these measures (see e.g.,~\cite[Section 1.4: Theorem 3]{Book:polyak1987}, \cite[Section 5.1]{primerryu} and \cite[Section 4.4]{lessard2014analysis}).
We also derive a tight worst-case guarantee for the proximal gradient method with exact line search.

Other related research trends feature linear convergence rates for the (proximal) gradient method under weaker assumptions than strong convexity, and linear convergence rates under inexact first-order information~\cite{schmidt2011convergence}. Among others, restricted strong convexity-type results are presented in~\cite{zhang2015restricted,necoara2015linear}, and convergence under {\color{red}{\L}ojasiewicz-type conditions were recently presented in~\cite{bolte2017error}}.

The paper is organized as follows: in Section~\ref{probstat}, we present the problem statement and the main results; in Section~\ref{secproofs}, we prove our main results (tight convergence rates for the method and for its exact line search variant); in Section~\ref{mixedperf}, we summarize known and newly derived tight results for the proximal gradient method, that were obtained using the performance estimation framework developed by Drori and Teboulle~\cite{Article:Drori} and the authors  (see~\cite{taylor2015smooth,taylor2015exact}). Finally, we conclude the work in Section~\ref{sec:ccl}.

\section{Problem Statement and Main Results}\label{probstat}
Let us consider the composite convex minimization (CCM) setting
\begin{equation}
 \min_{x\in\Rd} \left\{F(x)\defeq f(x)+h(x)\right\}\tag{CCM}\label{eq:SSOpt}
\end{equation}
where $h:\Rd\rightarrow\R\cup\{\pinf\}$ is convex, closed and proper and $f:\Rd\rightarrow\R$ is $L$-smooth $\mu$-strongly convex, closed and proper for some $0<\mu\leq L<\pinf$. That is, we assume that $f$ is differentiable over the whole $\Rd$ and satisfies the following conditions (using the Euclidean norm $\norm{.})$:
\begin{itemize}
\item $L$-smoothness:

$\norm{\nabla f(x+\Delta x)-\nabla f(x)}\leq L\norm{\Delta x}$ holds for all $x\in\Rd$ and $\Delta x\in\Rd$ ,
\item $\mu$-strong convexity:

$f(x)-\frac{\mu}{2}\normsq{x}$ is a convex function on $\Rd$.

\end{itemize}
In the sequel, we use the notation $\mathcal{F}_{0,\infty}(\Rd)$ to denote the class of closed and proper convex functions, and the notation $\mathcal{F}_{\mu,L}(\Rd)$ for the class of closed and proper $L$-smooth $\mu$-strongly convex functions.

In addition, we assume that we can evaluate the gradient of $f$ and the proximal operator of $h$~\cite[Section 1.1]{parikh2013proximal} at any $x \in \Rd$:
\begin{equation}
\prox{\gamma h}{x}\defeq\argmin{y\in\Rd}\left\{\gamma h(y)+\frac{1}{2}\normsq{x-y}\right\}.\tag{PROX}\label{eq:prox}
\end{equation} 
In this work, we focus on the proximal gradient method (PGM) with constant step length~$\gamma$ to solve~\eqref{eq:SSOpt}.\\[-.4cm]
\begin{center}
{\parbox{0.9\textwidth}{
        \textbf{Proximal gradient method (PGM)}
  \begin{itemize}
  \item[] Input: $x_0\in\Rd$, $f\in\mathcal{F}_{\mu,L}(\Rd)$, $h\in\mathcal{F}_{0,\infty}(\Rd)$, $0 \leq  \gamma \leq \frac{2}{L}$.
  \item[] For $k=0:N-1$
      \begin{align*}
    &x_{k+1}=\prox{\gamma h}{x_k-\gamma \nabla f(x_k)}
    \end{align*}
  \end{itemize}
       }}
\end{center}

The composite convex problem~\eqref{eq:SSOpt} admits the following very common particular cases:
\begin{itemize}
\item[-] the unconstrained minimization problem $\min_{x\in\R^d}  f(x)$, when $h(x)=0$ and $\prox{\gamma h}{x}=x$. In this case, PGM is simply the standard unconstrained gradient method (UGM) $x_{k+1}=x_k-\gamma\nabla f(x_k)$.
\item[-] The constrained minimization problem $\min_{x\in Q}  f(x)$, with $Q\subseteq \R^d$ a closed convex set. This corresponds to choosing $h(x)=i_Q(x)$ ($i_Q$ is the indicator function of $Q$) for which the proximal operation corresponds to a projection onto~$Q$: $\prox{\gamma i_Q}{x}=\Pi_Q (x)$. In this case, PGM is simply the standard projected gradient method ($\Pi$GM) $x_{k+1}=\Pi_Q(x_k-\gamma\nabla f(x_k))$. 
\item[-] The composite minimization problem $\min_{x\in\R^d} f(x)+h(x)$, where $h(x)$ has an analytical proximal operator available\footnote{A list of useful analytical proximal operators is available in~\cite{combettes2011proximal}.} (e.g., the classical $\ell_1$-regularization term $h(x)=\norm{x}_1$ and the corresponding soft-thresholding operator)).
\end{itemize}

\paragraph{Convergence rate} Consider Problem (CCM) with parameters $0<\mu\leq L<\infty$, and PGM with step size $\gamma$. Define the following quantity 
\begin{equation}
\rho(\gamma)\defeq\max\Bigl\{|1-L\gamma|,|1-\mu\gamma|\Bigr\}=\max\Bigl\{(L\gamma-1),(1-\mu\gamma)\Bigr\}\tag{RHO}\label{def:rho}.
\end{equation}
Observe that $\rho(\gamma)\geq 0$ for all values of the step size such that $0\leq \gamma\leq \frac{2}{L}$. The term $|1-L\gamma|$ in the expression of $\rho(\gamma)$ is minimized by taking the so-called \emph{short step} $1/L$, whereas the second term $|1-\mu\gamma|$ is minimized by choosing the so-called \emph{long step} $1/\mu$.

We prove in \secref{sec:tub} that applying PGM to Problem~\eqref{eq:SSOpt} produces a sequence of iterates converging to an optimal solution with a linear convergence rate, whatever the performance measure considered: \DOt, \RGNt\ and \OFAt. Moreover, that rate is equal to $\rho^2(\gamma)$ for all three measures. This implies that the best possible worst-case convergence rate is achieved by the step size $\rev{\gamma_\star} \defeq \frac{2}{L+\mu}$ for all three performance measures (see \figref{fig:rho}). Indeed, this step size minimizes $\rho(\gamma)$ (it corresponds to the case where both terms are equal), and the corresponding optimal rate is $\rev{\rho^{2}_{\star}} \defeq \rho^2(\rev{\gamma_\star}) = \bigl(\frac{L-\mu}{L+\mu}\bigr)^2$.

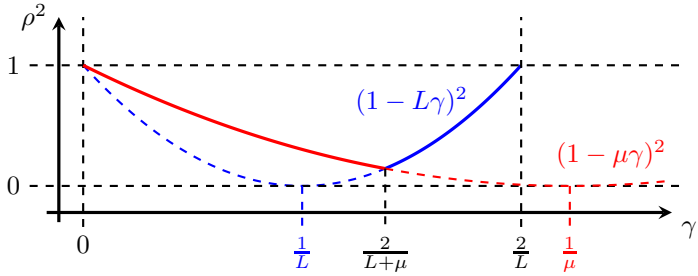
\begin{figure}[!ht]
\begin{center}
\begin{tikzpicture}[yscale=0.32,xscale=0.32]

\pgfmathsetmacro{\muv}{0.45};
\pgfmathsetmacro{\Lv}{1};
\pgfmathsetmacro{\AvLmu}{\muv/2+\Lv/2};
\pgfmathsetmacro{\rate}{(\Lv-\muv)^2/(\Lv+\muv)^2};

\draw[very thick, black, -stealth] (-10.5,-0.1) -- (15.2,-0.1);
\draw (15.2,-0.1) node[below right]{$\gamma$};
\draw[very thick, black, -stealth] (-10,-0.5) -- (-10,8);
\draw (-10.1,8) node[left]{$\rho^2$};
\draw(0,0) [very thick, color=blue, domain=9/\AvLmu-9:3*3, samples=50] plot({\x}, {5*(1-\Lv*(\x/9+1))^2+1});
\draw(0,0) [thick, dashed, color=blue, domain=-3*3:9/\AvLmu-9, samples=50] plot({\x}, {5*(1-\Lv*(\x/9+1))^2+1});
\draw[dashed, thick] (-9,-0.7) to (-9,8); 
\draw (-9,-0.7) node[below] {$0$};
\draw[dashed, thick] (9,-0.7) to (9,8); 
\draw (9,-0.7) node[below] {$\frac{2}{L}$};

\draw[dashed, thick] (-11.2,1) to (15.5,1);
\draw (-11.2,1) node[left] {$0$};

\draw[dashed, thick] (-11.2,6) to (15.5,6);
\draw (-11.2,6) node[left] {$1$};

\draw(0,0) [very thick, color=red, domain=-3*3:9/\AvLmu-9, samples=50] plot({\x}, {5*(1-\muv*(\x/9+1))^2+1});
\draw(0,0) [thick, dashed, color=red, domain=9/\AvLmu-9:15, samples=50] plot({\x}, {5*(1-\muv*(\x/9+1))^2+1});

\draw (12.7,2.3) node[red] {$(1-\mu\gamma)^2$};
\draw (4.5,4.5) node[blue] {$(1-L\gamma)^2$};

\draw[dashed, thick, red] (9/\muv-9,1) to (9/\muv-9,-0.7);
\draw (9/\muv-9,-0.7) node[below, red] {$\frac{1}{\mu}$};
\draw[dashed, thick, blue] (9/\Lv-9,1) to (9/\Lv-9,-0.7);
\draw (9/\Lv-9,-0.7) node[below, blue] {$\frac{1}{L}$};

\draw[dashed, thick] (9/\AvLmu-9,1+5*\rate) to (9/\AvLmu-9,-0.7);
\draw (9/\AvLmu-9,-0.7) node[below] {$\frac{2}{L+\mu}$};
\end{tikzpicture}
\caption{Rate of convergence $\rho^2(\gamma)$ as a function of the step size $\gamma$.}
\label{fig:rho}
\end{center}
\end{figure}

In the sequel, we will use notation $\tilde{\nabla}F(x)$ to denote a arbitrary subgradient of $F$ at $x$, i.e. $\tilde{\nabla}F(x)\in\partial F(x)$. It will also be convenient to define the following auxiliary quantity $s_k$ for every iteration of PGM  
\begin{equation}
 s_{k+1}\defeq\frac{x_k-x_{k+1}}{\gamma}-\nabla f(x_k).\label{def:sk}
\end{equation}
We will show in Section~\ref{sec:basic_ineq_and_notations} that $s_{k+1}\in\partial h(x_{k+1})$, i.e. that this quantity is a subgradient of $h$ (it is in fact the particular subgradient that appears in the optimality condition of the
proximal operator subproblem generating $x_{k+1}$). Another consequence is that $\nabla f(x_{k+1})+s_{k+1}\in\partial F(x_{k+1})$, i.e. $\nabla f(x_{k+1})+s_{k+1}$ is a particular subgradient of $F$ at $x_{k+1}$. Our results on the convergence of the residual gradient norm will actually characterize that particular subgradient.

Our main result can now be stated. \begin{theorem}
Let $f\in\mathcal{F}_{\mu,L}(\Rd)$ and $h\in\mathcal{F}_{0,\pinf}(\Rd)$, and consider the composite convex optimization problem~\eqref{eq:SSOpt} and a starting point $x_0\in\Rd$  satisfying $\partial F(x_0)\neq \emptyset$ (only required for the result in residual gradient norm). The iterates $\{ x_k \}_{k \ge 0}$ of PGM with step size $0\leq \gamma\leq \frac{2}{L}$ satisfy the following:
\begin{align*}
\begin{array}{lr}
\underset{\tiny\begin{array}{c}
f\in\mathcal{F}_{\mu,L}(\Rd)\\ h\in\mathcal{F}_{0,\pinf}(\Rd)\\ x_0\in\Rd
\end{array}}{\mathrm{max}} \left\{\frac{\normsq{x_{k}-x_\star}}{\normsq{x_0-x_\star}} \right\} &= \rho^{2k}(\gamma),\\
\underset{\tiny\begin{array}{c}
f\in\mathcal{F}_{\mu,L}(\Rd)\\ h\in\mathcal{F}_{0,\pinf}(\Rd) \\ x_0\in\Rd \end{array}}{\max} \left\{\frac{\normsq{\nabla f(x_{k})+s_{k}}}{||\tilde{\nabla}F(x_{0})||^2} \right\} &= \rho^{2k}(\gamma),\\
\underset{\tiny\begin{array}{c}
f\in\mathcal{F}_{\mu,L}(\Rd)\\ h\in\mathcal{F}_{0,\pinf}(\Rd) \\ x_0\in\Rd
\end{array}}{\max} \left\{\frac{F(x_{k})-F(x_\star)}{F(x_0)-F(x_\star)} \right\} &= \rho^{2k}(\gamma),
\end{array}
\end{align*}
where $x_\star\in\Rd$ denotes the (unique) optimal solution of~\eqref{eq:SSOpt}, $s_k$ denotes the subgradient used in the proximal operation to generate $x_k$ (\emph{see Equation~\eqref{def:sk}}), and $\tilde{\nabla}F(x_0)\in\partial F(x_0)$ denotes an arbitrary subgradient of $F$ at $x_0$.
\label{thm:conv_PGM_gen}
\end{theorem}

In addition, the following corollary directly uses Theorem~\ref{thm:conv_PGM_gen} to extend the recent results of~\cite[Theorem 1.2]{deKlerkELS2016} on the exact worst-case complexity of the gradient descent with exact line search. As in the unconstrained case ($h(x)=0$), this result cannot be improved in general, as it is attained by a two-dimensional quadratic example~\cite[Example 1.3]{deKlerkELS2016}.\\[-.4cm]
\begin{center}
	{\parbox{0.9\textwidth}{
			\textbf{Proximal gradient method with exact line search}
			\begin{itemize}
				\item[] Input: {$x_0\in\Rd$}, $f\in\mathcal{F}_{\mu,L}(\Rd)$, $h\in\mathcal{F}_{0,\infty}(\Rd)$.
				\item[] For $k=0:N-1$
				\begin{align*}
				&\gamma=\argmin{\gamma\in\R} F\left[\prox{\gamma h}{x_k-\gamma \nabla f(x_k)}\right]\\
				&x_{k+1}=\prox{\gamma h}{x_k-\gamma \nabla f(x_k)}
				\end{align*}
			\end{itemize}
		}}
	\end{center}
\begin{corollary}{Let $f\in\mathcal{F}_{\mu,L}(\Rd)$, and $h\in\mathcal{F}_{0,\pinf}(\Rd)$ and consider the composite convex optimization problem~\eqref{eq:SSOpt} and some starting point $x_0\in\Rd$. The iterates of PGM with exact line search satisfy the following inequality:
\[ F(x_{k+1})-F_\star \leq \rho_\star^2 \left(F(x_{k})-F_\star\right).\]
\label{thm:convg_PGM_ELS_f}}
\end{corollary}
\begin{proof} This is exactly the result of Theorem~\ref{thm:conv_PGM_gen} (in terms of objective function accuracy) for the choice $\gamma=\gamma_\star$. The corresponding result is an upper bound on the worst-case of PGM with exact line search, which turns out to be tight on the quadratic example~\cite[Example 1.3]{deKlerkELS2016}~\cite[Example on p.69]{Book:Bertsekas_nlp}.
\end{proof}

\begin{remark}\label{remark:fromref}
As suggested by an anonymous referee, it is possible to perform exact line searches over other performance measures, and to deduce results similar to that of Corollary~\ref{thm:convg_PGM_ELS_f} for those methods.  For example, in the simple case $h(x)=0$, it directly follows from Theorem~\ref{thm:conv_PGM_gen} that the method using an exact line search over the residual gradient norm
\[ \alpha=\mathrm{argmin}_\alpha \{ \norm{\nabla f(x_{k}+\alpha\nabla f(x_k))}\}\]
and $x_{k+1}=x_k+\alpha\nabla f(x_k)$ satisfies
\[ \norm{\nabla f(x_{k+1})}^2\leq \rho_\star^2\norm{\nabla f(x_k)}^2,\]
and a similar result holds for the distance to optimality. The corresponding results for the composite case $h(x)\neq 0$ can also easily be deduced from the same reasoning. However, we do not know whether the upper bounds valid for those measures (with line search) are tight, due to a lack of matching lower bounds.
\end{remark}
\paragraph{Prior works} The UGM and $\Pi$GM are standard methods, whose analysis in the context of smooth strongly convex functions can be found in numerous references. Linear convergence in \DOt\ according to
\begin{equation}
\norm{x_{k+1}-x_\star}^2 \leq \rho^2(\gamma) \norm{x_{k}-x_\star}^2,\label{eq:prior_work1}
\end{equation}
can be found in e.g.,~\cite[Section 1.4: Theorem 3]{Book:polyak1987}, \cite[Section 5.1]{primerryu} and \cite[Section 4.4]{lessard2014analysis} for UGM and $\Pi$GM, with slight variations in the assumptions (depending on whether or not $f$ is required to be twice differentiable). For the specific step size $1/L$, the guarantee~\eqref{eq:prior_work1} can be found as a particular case of~\cite[Proposition 3]{schmidt2011convergence}. Weaker convergence rates such as $\rho(1/L)=(1-\frac{\mu}{L})$ for the specific step size $\gamma=1/L$ can be found in e.g.,~\cite[Theorem 2.2.8]{Book:Nesterov} or~\cite[Theorem 3.10]{bubeck2014theory} for $\Pi$GM. One can also check that~\eqref{eq:prior_work1} also holds for PGM, as it essentially follows the same proof technique as for $\Pi$GM (using the non-expansiveness of the proximal operation).

As far as we know, results in terms of \OFA\ or \RGN\ are typically not {as} emphasized (or known) as compared to convergence in terms of \DO. A standard technique to convert convergence results in terms of \DO\ to \RGN\ and \OFA\ consists in exploiting the smoothness and strong convexity assumptions (convergence in \RGN\ and in \OFA\ is then obtained as a by-product of the convergence in \DO). For example, in the {particular} case of unconstrained minimization (i.e., $h=0$), one can use:
\begin{align*}
f(x_k)-f(x_\star)\leq \frac{L}{2} \normsq{x_k-x_\star},\ \norm{\nabla f(x_k)}\leq L \norm{x_k-x_\star},\\
f(x_k)-f(x_\star)\geq \frac{\mu}{2} \normsq{x_k-x_\star},\ \norm{\nabla f(x_k)}\geq \mu \norm{x_k-x_\star},
\end{align*}
to adapt convergence rates in terms of \DOt\ to the following convergence results in \OFAt\ and \RGNt:
\begin{equation}
f(x_k)-f(x_\star)\leq \frac{L}{\mu}\rho^{2k}{(\gamma)} (f(x_0)-f(x_\star)), \text{ and } \norm{\nabla f(x_k)}\leq  \frac{L}{\mu}\rho^k{(\gamma)} \norm{\nabla f(x_0)}.\label{eq:weak_guarantee}
\end{equation}
However, the bounds~\eqref{eq:weak_guarantee} are not tight because of the leading constant $L/\mu$ (see Theorem~\ref{thm:conv_PGM_gen}). In addition,  we typically have $\frac{L}{\mu}\rho^{2k}>1$ for small values of $k$, and therefore the former inequality~\eqref{eq:weak_guarantee} does not even guarantee an improvement in terms of \OFAt\ or \RGNt\ for few iterations.

The global convergence rate $\rho^2\left(\rev{\gamma_\star}\right)=\left(\frac{L-\mu}{L+\mu}\right)^2$ in terms of \OFAt\ was only very recently obtained for UGM with step size $\rev{\gamma_\star}$ as a by-product of the convergence guarantee when using exact line search for solving unconstrained smooth convex minimization problems~\cite[Theorem 1.2]{deKlerkELS2016}, whereas previous results were only establishing a $\left(1-\frac{\mu}{L}\right)$ global convergence rate (see e.g.,~\cite{Book:Bertsekas_nlp}). Theorem~\ref{thm:conv_PGM_gen} further extends this result in the composite case~\eqref{eq:SSOpt} for the different convergence measures, for embedding a projection or a proximal step and for all reasonable step sizes.

\section{Convergence in Function Values, Residual Gradient and Distance to Optimum}\label{secproofs}
\subsection{Quadratic Lower Bounds}
\label{sec:quadLB}

First, we focus on the case of $f$ being quadratic without $h$ ($h=0$), which provides us with lower complexity bounds for the different values of the step size $\gamma$. Those quadratics correspond to lower bounds for UGM and therefore also for $\Pi$GM and PGM. We will show that those are tight for the class of smooth strongly convex functions in the following section.

Consider two constants $0< \mu \leq L < +\infty$ and the corresponding quadratic functions $f_\mu(x)\defeq\frac{\mu}{2} \normsq{x}$ and $f_L(x)\defeq\frac{L}{2}\normsq{x}$. We clearly have that $f_\mu,f_L\in\mathcal{F}_{\mu,L}(\Rd)$ and that $x_\star=0$ and $f_\star=0$ for both functions. In addition, one iteration of UGM on those functions respectively gives:
\begin{align*}
x_{k+1}^{(\mu)}=(1-\gamma \mu)x_k^{(\mu)}, \quad x_{k+1}^{(L)}=(1-\gamma L)x_k^{(L)},
\end{align*}
which respectively lead to 
\begin{equation*}
\begin{array}{ll}
 f_\mu(x_{k+1}^{(\mu)})=\frac{\mu}{2} (1-\gamma \mu)^2\normsq{x_k^{(\mu)}}, \quad &  \normsq{\nabla f_\mu(x_{k+1}^{(\mu)})}=(1-\gamma \mu)^2\normsq{\nabla f_\mu(x_k^{(\mu)})}, \\
f_L(x_{k+1}^{(L)})=\frac{L}{2} (1-\gamma L)^2\normsq{x_k^{(L)}},& {\normsq{\nabla f_L(x_{k+1}^{(L)})}}=(1-\gamma L)^2\normsq{\nabla f_L(x_k^{(L)})}.
\end{array}
\end{equation*}
Those equalities allow to conclude that the worst-case behaviour for any of the criteria $f(x_{k+1})-f_\star$, $\norm{x_{k+1}-x_\star}^2$ and $\norm{\nabla f(x_{k+1})}^2$ is at least as bad as in the cases of those two functions. That is, for any $\gamma\in\mathbb{R}$, there exists a $f\in\mathcal{F}_{\mu,L}(\Rd)$ such that all iterations ($k\geq 0$) of UGM statisfy:
\begin{align*}
\normsq{x_{k+1}-x_\star}&= \rho^{2}(\gamma) \normsq{x_{k}-x_\star},\\
f(x_{k+1})-f_\star&= \rho^{2}(\gamma)(f(x_{k})-f_\star),\tag{QLB}\label{eq:quad_lowerbounds}\\
\normsq{\nabla f(x_{k+1})} &= \rho^{2}(\gamma) \normsq{\nabla f(x_{k})}.
\end{align*}
We will see that that no other function within $f\in\mathcal{F}_{\mu,L}(\Rd)$ behaves (strictly) worse in the cases of interest.

\subsection{Basic Inequalities Characterizing one Iteration of PGM}
\label{sec:basic_ineq_and_notations}
Now, we make a short inventory of the inequalities available to prove the different global convergence rates. Recent works on performance estimation of first-order methods (see~\cite{taylor2015smooth,taylor2015exact})  guarantee that no other inequalities are needed in order to obtain the desired convergence results.

In the following, we denote by $g_k$ and $s_k$ the (sub)gradients of respectively the smooth function $f$ and the non-smooth component $h$ at the iteration $k$; that is $g_k\defeq\nabla f(x_k)$ and $s_k\in\partial h(x_k)$, and by $f_k$ and $h_k$ the function values at those points: $f_k\defeq f(x_k)$ and $h_k\defeq h(x_k)$. In addition to that, we denote by $x_\star$ the optimal point (unique by strong convexity of $F$) and by $g_\star\defeq\nabla f(x_\star)$ and $s_\star\in\partial h(x_\star)$ the gradient and some subgradient of respectively $f$ and $h$ at the optimum.
Let us list the (in)equalities that enable us to characterize one iteration of PGM.
\begin{itemize}
\item[(a)] The iteration $x_{k+1}=\prox{\gamma h}{x_k-\gamma\nabla f(x_k)}$ can be rewritten using necessary and sufficient optimality conditions on the definition of the proximal operation~\eqref{eq:prox}:\[x_{k+1}=x_k-\gamma (g_k+s_{k+1})\] for some $s_{k+1}\in\partial h(x_{k+1})$.
\item[(b)] Optimality of $x_\star$ for~\eqref{eq:SSOpt} amounts to requiring $g_\star+s_\star=0$ for some $s_\star\in\partial h(x_\star)$. 
\item[(c)] For characterizing smoothness and strong convexity in the case $0<\mu<L<\infty$, we use the conditions from~\cite[Theorem 4]{taylor2015smooth}. This should be required between three points: $x_{k}, x_{k+1}$ and $x_\star$. That is, 
for the six possible pairs $(i,j)$ within $\left\{k,k+1,\star\right\}$ 
we have:
\begin{align}
\begin{array}{ll}
f_i \geq  f_j &+ \inner{g_j}{x_i-x_j}+\frac{1}{2L}\normsq{g_i-g_j}\\
&+\frac{\mu}{2(1-\mu/L)}\normsq{x_i-x_j-\frac{{1}}{L}(g_i-g_j)}.
\end{array}
 \label{eq:interp_smooth}
\end{align}
Note that in the remaining case $L=\mu$,  function $f$ has to be of the form $f(x)=\frac L 2\normsq{x-x_\star}$, and hence \rev{all} results can easily be verified separately for this case.
\item[(d)]  Similarly, we require that for the same six pairs $(i,j)$:
\begin{align}
h_i-h_j-\inner{s_j}{x_i-x_j}\geq 0,\label{eq:interp_nonsmooth}
\end{align}
for characterizing the (possibly non-smooth) convex function $h$. 
\end{itemize}

\subsection{Upper Bounds on the Global Convergence Rates}
\label{sec:tub}

In this section, we prove the main convergence results of the paper, beginning with the convergence in terms of \DOt. Each proof of this section is provided with a symbolic verification code (link available at the end of Section~\ref{sec:ccl}).

\paragraph{Distance to optimality} As mentioned in Section~\ref{probstat}, the following convergence result in term of \DOt\ is not new. For the sake of clarity and completeness, we start by proving it using the same technique as for the subsequent results (on \RGNt\ and \OFAt). The proof methodology relies from the performance estimation methodology (see~\cite{Article:Drori,taylor2015smooth,taylor2015exact,deKlerkELS2016,drori2014contributions}).
 This technique has the advantage of being transparent and of explicitly identifying \emph{weaker assumptions} for obtaining this convergence property (see discussions below Theorem~\ref{thm:convg_GM_x}).
\begin{theorem}[Distance to optimality]
Consider the composite convex optimization problem~\eqref{eq:SSOpt}. Every pair of consecutive iterates of the PGM with $0\leq \gamma\leq \frac{2}{L}$ satisfies the following inequality:
\[ \normsq{x_{k+1}-x_\star} \leq\rho^2(\gamma) \normsq{x_k-x_\star}.\]
\label{thm:convg_GM_x}
\end{theorem}
\begin{proof}
We use the notations and inequalities introduced in the previous section (Section~\ref{sec:basic_ineq_and_notations}) in order to construct the proof. As proposed in Section~\ref{sec:basic_ineq_and_notations}, we use some of the inequalities~\eqref{eq:interp_smooth} and~\eqref{eq:interp_nonsmooth} between the iterates and the optimal point. The proof consists in summing those interpolation inequalities after multiplying them with their respective coefficients\footnote{Those $\lambda$'s were found by identifying an analytical optimal solution to the dual performance estimation problem. That is, each $\lambda$ can be seen as a Lagrange multiplier for the corresponding inequality. The methodology is explained and illustrated in details in~\cite[Section 4.1]{deKlerkELS2016}.}, or multipliers, $\lambda$'s.

First, we use~\eqref{eq:interp_smooth} with respectively $(i,j)=(\star,k)$ and $(i,j)=(k,\star)$:
\begin{align*}
&&\begin{array}{ll}
f_\star\geq f_k&+\inner{g_k}{x_\star-x_k}+\frac{1}{2L}\normsq{g_k-g_\star} \\ & +\frac{\mu}{2(1-\mu/L)}\normsq{x_k-x_\star-\frac{{1}}{L}(g_k-g_\star)}
\end{array}
\quad &:\lambda_0,  \\
&&\begin{array}{ll}
f_k\geq f_\star&+\inner{g_\star}{x_k-x_\star}+\frac{1}{2L}\normsq{g_k-g_\star} \\ &+\frac{\mu}{2(1-\mu/L)}\normsq{x_k-x_\star-\frac{{1}}{L}(g_k-g_\star)}
\end{array}
\quad &:\lambda_1
\end{align*}
Then, we use~\eqref{eq:interp_nonsmooth} with respectively $(i,j)=(\star,k+1)$ and $(i,j)=(k+1,\star)$:
\begin{align*}
&& h_\star\geq h_{k+1}&+\inner{s_{k+1}}{x_\star-x_{k+1}} &:\lambda_2,\\
&& h_{k+1}\geq h_\star &+ \inner{s_\star}{x_{k+1}-x_\star}&:\lambda_3.
\end{align*}We use the following multipliers
\begin{align*}
\lambda_0=\lambda_1=2\gamma\rho(\gamma)\geq 0, \quad \lambda_2=\lambda_3=2\gamma\geq 0.
\end{align*}After appropriate substitutions of $x_{k+1}$ and $s_\star$, using $x_{k+1}=x_k-\gamma (g_k+s_{k+1})$ (Section~\ref{sec:basic_ineq_and_notations}, Condition (a)) and $s_\star=-g_\star$ (Section~\ref{sec:basic_ineq_and_notations}, Condition (b)), and with some effort, one can check that the previous weighted sum of inequalities can be written in one of the following forms. We divide the proof in two cases ({corresponding to the two regimes of $\rho(\gamma)$, see~\eqref{def:rho}}). In order to illustrate the reformulation procedure, verification for the first case is detailed explicitly in \ref{app:proofhelper}.
\begin{itemize}
\item  When $0\leq \gamma\leq \frac{2}{L+\mu}$ (i.e., $\rho(\gamma)=(1-\gamma\mu)$), the  expression can be written as (details in~\ref{app:proofhelper})
\begin{align*}
 \left(1-\gamma \mu \right)^2 \normsq{x_k-x_\star} \geq &\normsq{x_{k+1}-x_\star}  +\gamma^2 \normsq{g_\star+s_{k+1}}\\&+\frac{\gamma(2-\gamma (L+\mu))}{L-\mu} \normsq{\mu {(x_k  -x_\star)} - g_k+  g_\star},\\
 \geq &\normsq{x_{k+1}-x_\star},
\end{align*} where the last inequality follows from \[\gamma^2\geq 0,\ \gamma(2-\gamma (L+\mu))\geq 0,  \text{ and } L-\mu\geq 0.\]
\item Similarly, when $\frac{2}{L+\mu} \leq \gamma\leq \frac{2}{L}$ (i.e., $\rho(\gamma)=( L\gamma-1)$), the  expression is equivalent to
 \begin{align*}
 \left(1-\gamma L \right)^2 \normsq{x_k-x_\star} \geq &\normsq{x_{k+1}-x_\star} +\gamma^2 \normsq{g_\star+s_{k+1}}\\&+\frac{\gamma(\gamma (L+\mu)-2)}{L-\mu} \normsq{L{(x_k -x_\star)}- g_k +  g_\star},\\
 \geq &\normsq{x_{k+1}-x_\star},
\end{align*} where the last inequality follows from \[\gamma^2\geq 0, \ \gamma(\gamma (L+\mu)-2)\geq 0, \text{ and } L-\mu\geq 0.\]
\end{itemize}
Note that for any $\gamma$ such that $0\leq \gamma \leq \frac{2}{L}$, exactly\footnote{Actually, both regimes are valid for $\gamma=\frac{2}{L+\mu}$.} one of the two previous combinations of inequalities is valid (both multipliers and coefficients of the squared norms are positive). In addition, the valid expression corresponds to the maximum value between the two possible rates $(1-\gamma \mu)^2$ and $(1-\gamma L)^2$, which concludes the proof.
\end{proof}
Before moving to the next convergence result, note that only a subset of the available inequalities were used in the previous proof. In fact, any composite function $F$ for which the $f$ component satisfies $\forall x\in\Rd$:
\begin{align}
 \label{eq:relax_dist}\inner{\nabla f(x_\star)-\nabla f(x)}{x_\star-x}&\geq \frac{1}{L}\normsq{\nabla f(x)-\nabla f(x_\star)}\\ \notag &+\frac{\mu}{1-\mu/L}\normsq{x-x_\star-\frac{{1}}{L}(\nabla f(x)-\nabla f(x_\star))}, \end{align}(which is the sum of the two first inequalities used in the proof of Theorem~\ref{thm:convg_GM_x}, as equal multipliers $\lambda_0=\lambda_1$ are used) will have a PGM converging with the same rate in terms of {\DOt}, despite being potentially outside of $\mathcal{F}_{\mu,L}$. For example, consider the following quadratic function $f_A(x)=\frac{1}{2} x^\top A x$ with $\mu I\preceq A \preceq L I$ (hence $f\in\mathcal{F}_{\mu,L}$). Therefore,~\eqref{eq:relax_dist} holds and hence:
\begin{align*}
\left(1+\frac{\mu}{L}\right)(x_\star-x_k)^\top A(x_\star-x_k)\geq &\frac{1}{L}(x_k-x_\star)^\top A^\top A(x_k-x_\star)\\&+\mu(x_k-x_\star)^\top (x_k-x_\star).
\end{align*}
However, this inequality also holds in some cases where $0\preceq A \preceq L I$. For example, let $x_\star$ be the projection of $x_k$ onto the set of optimal solutions (i.e., $x_k-x_\star \perp \mathrm{Null}(A)$), and let $\mu>0$ correspond to the smallest non-zero eigenvalue of~$A$. In that case, the previous inequality is satisfied, although $f_A$ does not belong to $\mathcal{F}_{\mu,L}$, and the gradient method for minimizing $f_A$ converges with the same rate $\rho^2(\gamma)$ in terms of distance to an optimal point.

Also, we only need to require a monotonicity condition on $\partial h$ for keeping the same convergence guarantees, as only the sum of the third and fourth inequalities is used in the proof ($\lambda_2=\lambda_3$):
\[\inner{s_{k+1}-s_\star}{x_{k+1}-x_\star}\geq 0,\]
Those sorts of relaxations were further exploited in~\cite{zhang2015restricted,necoara2015linear} (relaxation of the strong convexity requirement, with motivational examples). We leave further investigations in that direction for future research.

\paragraph{Residual gradient norm} 
The next theorem is concerned with convergence in terms of \RGNt. Similar results can be obtained for the norm of the (composite) gradient mapping {(i.e.,$\frac{x_{k}-x_{k+1}}{\gamma}$)} instead\footnote{The difference between the gradient mapping and the \RGNt\ is simple, but somewhat subtle. The gradient mapping measures $\norm{\nabla f(x_k)+s_{k+1}}$, whereas the \rev{residual gradient norm} measures $\norm{\nabla f(x_{k+1})+s_{k+1}}$ with $s_{k+1}\in\partial h(x_{k+1})$ the subgradient used in the proximal operation.}, which is used in some standard references on composite minimization~\cite{Book:Nesterov,nesterov2007gradient}.

Convergence in terms of \RGNt\ is in fact {very natural}, as it is measurable in practice, as opposed to the \DOt, which requires the knowledge of $x_\star$ in order to be evaluated, or in terms of \OFAt, which requires the knowledge of (or a least a bound on) the true value of $F(x_\star)$.

\begin{theorem}[Residual gradient norm]
Let $f\in\mathcal{F}_{\mu,L}(\Rd)$, and $h\in\mathcal{F}_{0,\pinf}(\Rd)$, and consider the composite convex optimization problem~\eqref{eq:SSOpt}, and a feasible starting point $x_0\in\Rd$ (i.e., $x_0$ is such that $F(x_0)<\infty$) such that there exists $s_0\in\partial h(x_0)$. The iterates of PGM with $0\leq \gamma\leq \frac{2}{L}$ satisfy:
\[ \normsq{\nabla f(x_{k+1})+s_{k+1}} \leq \rho^2(\gamma) \normsq{\nabla f(x_k)+s_k},\]
with $s_k\in\partial h(x_k)$  (any subgradient of $h$ at $x_k$) and $s_{k+1}\in\partial h(x_{k+1})$, the subgradient of $h$ at $x_{k+1}$ used in the proximal operation (\emph{see Equation~\eqref{def:sk}}).
\label{thm:convg_GM_G}
\end{theorem}
\begin{proof}
We use the exact same reasoning as for Theorem~\ref{thm:convg_GM_x}: the
notations and inequalities are introduced in the previous section (Section~\ref{sec:basic_ineq_and_notations}), and the proof consists in summing the following interpolation inequalities after multiplication by their respective coefficients. The main difference lies in the choice of the inequalities to be combined. In this proof, we use conditions between consecutive iterates, instead of using conditions between the current iterates and the optimum:
{\begin{align*}
&\begin{array}{ll}
f_k\geq f_{k+1}&+\inner{g_{k+1}}{x_k-x_{k+1}}+\frac{1}{2L}\normsq{g_k-g_{k+1}}\\ &+\frac{\mu}{2(1-\mu/L)}\normsq{x_k-x_{k+1}-\frac{1}{L}(g_k-g_{k+1})} 
\end{array}\quad &:\lambda_0\\
&\begin{array}{ll}
f_{k+1}\geq f_k&+\inner{g_k}{x_{k+1}-x_k}+\frac{1}{2L}\normsq{g_k-g_{k+1}}\\& +\frac{\mu}{2(1-\mu/L)}\normsq{x_k-x_{k+1}-\frac{1}{L}(g_k-g_{k+1})} 
\end{array}
&:\lambda_1, \\
&h_k\geq h_{k+1} + \inner{s_{k+1}}{x_k-x_{k+1}}&:\lambda_2,\\
&h_{k+1}\geq h_k + \inner{s_k}{x_{k+1}-x_k}&:\lambda_3.
\end{align*}}We use the following multipliers:
\begin{align*}
\lambda_0=\lambda_1=\frac{2}{\gamma}\rho(\gamma)\geq 0,\quad
\lambda_2=\lambda_3=\frac{2}{\gamma}\rho^2(\gamma)\geq 0.
\end{align*}
After appropriate substitutions of $x_{k+1}$, using $x_{k+1}=x_k-\gamma (g_k+s_{k+1})$ (Section~\ref{sec:basic_ineq_and_notations}, Condition (a)), the previous weighted sum corresponds to a sum of squares in the two cases of interest ({same two regimes as $\rho(\gamma)$, see~\eqref{def:rho}}).
\begin{itemize}
\item When $0\leq \gamma\leq \frac{2}{L+\mu}$ (i.e., when $\rho(\gamma)=(1-\gamma\mu)$):
\begin{align*}
\left(1-\gamma \mu \right)^2 \normsq{g_k+s_k} \geq & \normsq{g_{k+1}+s_{k+1}} +\left(1-\gamma \mu \right)^2\normsq{s_k-s_{k+1}}\\ &+
\frac{2-\gamma (L+\mu)}{\gamma (L-\mu)} \normsq{g_k-g_{k+1}-\mu\gamma(g_k+s_{k+1})},\\
\geq &\normsq{g_{k+1}+s_{k+1}},
\end{align*} where the last inequality follows from 
\[ (1-\gamma \mu)^2\geq 0,\ 2-\gamma (L+\mu)\geq 0, \text { and } \gamma(L-\mu)\geq 0.\]
\item When $\frac{2}{L+\mu} \leq \gamma\leq \frac{2}{L}$ (i.e., when $\rho(\gamma)=(L\gamma -1)$):
 \begin{align*}
\left(1-\gamma L \right)^2 \normsq{g_k+s_k} \geq & \normsq{g_{k+1}+s_{k+1}} +\left(1-\gamma L \right)^2\normsq{s_k-s_{k+1}}\\&+
\frac{\gamma (L+\mu)-2}{\gamma (L-\mu)} \normsq{g_k-g_{k+1}-L\gamma(g_k+s_{k+1})},\\
\geq &\normsq{g_{k+1}+s_{k+1}},
\end{align*} and the last inequality follows from
\[(1-\gamma L)^2\geq 0,\ \gamma (L+\mu)-2\geq 0, \text{ and } \gamma(L-\mu)\geq 0.\]
\end{itemize}

We conclude the proof in the same way as for the distance to optimality: since, for any value of $\gamma$ such that $0\leq \gamma \leq \frac{2}{L}$, there is always one of the two previous combinations of inequalities that is valid (both multipliers and coefficients of the squared norms are positive), and since the valid one corresponds to the maximum value between the two possible rates $(1-\gamma \mu)^2$ and $(1-\gamma L)^2$, the desired statement is proved. \end{proof}
Interestingly, the inequalities used in this proof do not involve the optimal point, and only use the {information} available at the consecutive iterates. In addition, {note} that {as for the convergence in terms of \DOt,} $\lambda_0=\lambda_1$ tells us that the result hold\rev{s} under the following weaker assumption:
\begin{align*}
\inner{g_{k+1}-g_k}{x_{k+1}-x_k}\geq &\frac{1}{L}\normsq{g_k-g_{k+1}}\\&+\frac{\mu}{1-\mu/L}\normsq{x_k-x_{k+1}-\frac{{1}}{L}(g_k-g_{k+1})}.
\end{align*}
A consequence of this inequality is that one can benefit from using the locally better strong convexity and smoothness parameters (i.e., better constants $\mu$ and $L$ that satisfy this inequality for two consecutive iterates) instead of the global ones, in order to improve the convergence rate. {Also, it is possible to exploit this in order to make online estimations of the strong convexity and smoothness parameters $\mu$ and $L$} (we leave this for further research).

\paragraph{Objective function accuracy} 
{Finally, we consider convergence} in terms of \OFAt. The proof of this convergence rate is much more tedious than the previous ones, and seems to require more assumptions {(i.e., more inequalities appear to be needed --- of course it may also be that we just did not discover the simplest proof).}
 
\begin{theorem}[Objective function accuracy]
Let $f\in\mathcal{F}_{\mu,L}(\Rd)$, and $h\in\mathcal{F}_{0,\pinf}(\Rd)$, and consider the composite convex optimization problem~\eqref{eq:SSOpt}, and some starting point $x_0\in\Rd$. The iterates of PGM with $0\leq \gamma\leq \frac{2}{L}$ satisfy the following:
\[ F(x_{k+1})-F_\star \leq \rho^2(\gamma) \left(F(x_{k})-F_\star\right).\]
\label{thm:convg_GM_f}
\end{theorem}
\begin{proof}
We combine the following interpolation after multiplication with their respective coefficients:
 \begin{align*}
 &\begin{array}{ll}
f_k\geq f_{k+1}&+\inner{g_{k+1}}{x_k-x_{k+1}}+\frac{1}{2L}\normsq{g_k-g_{k+1}}\\ &+\frac{\mu}{2(1-\mu/L)}\normsq{x_k-x_{k+1}-\frac{1}{L}(g_k-g_{k+1})}  
 \end{array}
&:\lambda_0,\\
&\begin{array}{ll}
f_\star\geq f_k&+\inner{g_k}{x_\star-x_k}+\frac{1}{2L}\normsq{g_k-g_\star}\\ &+\frac{\mu}{2(1-\mu/L)}\normsq{x_k-x_\star-\frac{1}{L}(g_k-g_\star)} 
\end{array}
&:\lambda_1, \\
&\begin{array}{ll}
f_\star\geq f_{k+1}&+\inner{g_{k+1}}{x_\star-x_{k+1}}+\frac{1}{2L}\normsq{g_\star-g_{k+1}}\\
&+\frac{\mu}{2(1-\mu/L)}\normsq{x_\star-x_{k+1}-\frac{1}{L}(g_\star-g_{k+1})}
\end{array}
 &:\lambda_2,\\
&h_k\geq h_{k+1}+\inner{s_{k+1}}{x_k-x_{k+1}}&:\lambda_3,\\
&h_\star\geq h_{k+1} + \inner{s_{k+1}}{x_\star-x_{k+1}}&:\lambda_4.
\end{align*} We use the following multipliers:
\begin{align*}
\lambda_0=\rho(\gamma),\ \lambda_1=(1-\rho(\gamma))\rho(\gamma), \
\lambda_2=1-\rho(\gamma), \ \lambda_3=\rho^2(\gamma),\
 \lambda_4=1-\rho^2(\gamma).
\end{align*}Appropriate substitutions of $x_{k+1}$ and $s_\star$ using $x_{k+1}=x_k-\gamma (g_k+s_{k+1})$ (Section~\ref{sec:basic_ineq_and_notations}, Condition (a)) and $s_\star=-g_\star$ (Section~\ref{sec:basic_ineq_and_notations}, Condition (b)), we obtain that the weighted sum of inequalities is equivalent to the following expressions.
\begin{itemize}
\item When $0\leq \gamma\leq \frac{2}{L+\mu}$ (i.e., when $\rho(\gamma)=(1-\gamma\mu)$):
{\small
\begin{align*}
&\left(1-\gamma \mu\right)^2 \left(F(x_k)-F_\star\right) \\ &\geq  F(x_{k+1})-F_\star+\frac{(2-\gamma\mu)\beta(\gamma)}{2\alpha(\gamma)}\normsq{(1-\gamma\mu)g_k-g_{k+1}+\mu\gamma g_\star}\\
&+\frac{\gamma L\mu^2(2-\gamma\mu)}{2(L-\mu)}\normsq{(x_k-x_\star)-\frac{2L-2\mu+\gamma \mu^2}{L\mu(2-\gamma\mu)}s_{k+1}-\frac{g_k+g_{k+1}}{\mu (2-\gamma\mu)}+\frac{g_\star}{L}}\\
&+\frac{\gamma \mu\alpha(\gamma)}{2L(L-\mu)(2-\gamma\mu)}\normsq{ {\color{red}s_{k+1}}+\frac{(\mu\gamma-1)L\beta(\gamma)}{\alpha(\gamma)}g_k+\frac{L\beta(\gamma)}{\alpha(\gamma)}g_{k+1}+\frac{(L-\mu)(2-\gamma \mu)^2}{\alpha(\gamma)} g_\star},\\
&\geq F(x_{k+1})-F_\star,
\end{align*}}with $\alpha(\gamma)\defeq-(\gamma^2 L^2 \mu + 2L (-2+\gamma\mu)+\mu (-2+\gamma \mu)^2)$ and $\beta(\gamma)\defeq(2-\gamma(L+\mu))$, and where the last inequality follows from the signs of the leading coefficients. Indeed, note that $\alpha(\gamma)$ is positive for $0\leq \mu < L$ and $0\leq\gamma \leq \frac{2}{\mu+L}$, as $\alpha(\gamma)$ is  a concave quadratic function with $\alpha(0)\geq 0$ and $\alpha(\rev{\gamma_\star})=\frac{4L^2(L - \mu)}{(L + \mu)^2}\geq 0$.
\item When $\frac{2}{L+\mu} \leq \gamma\leq \frac{2}{L}$ (i.e., when $\rho(\gamma)=(L\gamma -1)$):
{\small \begin{align*}
&\left(1-\gamma L\right)^2 \left(F(x_k)-F_\star\right)\\& \geq  F(x_{k+1})-F_\star+\frac{(2-\gamma L)\beta(\gamma)}{2\gamma \alpha(\gamma)}\normsq{(1-\gamma L)g_k-g_{k+1}+ \gamma L g_\star}\\
&+\frac{\gamma L^2 \mu(2-\gamma L)}{2(L-\mu)}\normsq{(x_k-x_\star)-\frac{s_{k+1}}{\mu}+\frac{1-\gamma L-\gamma \mu}{\gamma L\mu}g_k-\frac{g_{k+1}}{\gamma L\mu}+\frac{g_\star}{L}}\\
&+\frac{\gamma \alpha(\gamma)}{2\mu (L-\mu)}\normsq{ s_{k+1}+\frac{(\gamma L-1)L\beta(\gamma)}{\gamma \alpha(\gamma)}g_k+\frac{L\beta(\gamma) }{\gamma \alpha(\gamma)}g_{k+1}+\frac{(2-\gamma L)(L-\mu)\mu}{\alpha(\gamma)} g_\star},\\
&\geq F(x_{k+1})-F_\star,
\end{align*}}with $\alpha(\gamma)\defeq(-2L^2-2\mu^2+2L\mu+\gamma L^3+\gamma L\mu^2)$ and $\beta(\gamma)\defeq(\gamma(L+\mu)-2)$, and where the last inequality follows from the signs of the leading coefficients. Again, $\alpha(\gamma)$ is nonnegative as it is an increasing linear function which is nonnegative in the region of interest $\gamma\geq \frac{2}{L+\mu}$. Indeed, one can check that by evaluating \rev{$\alpha(\gamma_\star)=2\mu^2\left(\frac{L-\mu}{L+\mu}\right)\geq 0$}.
\end{itemize}
We conclude the proof in the same way as before: among the two cases, the valid one corresponds to the maximum value between the two possible rates $(1-\gamma \mu)^2$ and $(1-\gamma L)^2$.
\end{proof}
\paragraph{Proof of Theorem~\ref{thm:conv_PGM_gen}} The Theorem is obtained by combining the upper bounds from Theorem~\ref{thm:convg_GM_x}, Theorem~\ref{thm:convg_GM_G} and Theorem~\ref{thm:convg_GM_f} with the lower bounds provided by the quadratic functions~\eqref{eq:quad_lowerbounds}.\qed

\section{\rev{Mixed Performance Measures and Sublinear Convergence Rates}} \label{mixedperf}
In this section, we consider \emph{mixed performance measures} for PGM, that is, we consider cases where the quality of the initial and final iterates $x_0$ and $x_k$ are measured using different criteria. For example, the quality of $x_0$ may be measured in terms of $\norm{x_0-x_\star}$ whereas the quality of $x_k$ may be measured in terms of $F(x_k)-F(x_\star)$. As we will see, this type of combination is particularly useful for obtaining \emph{sublinear} convergence rates in the smooth convex but \rev{non-strongly} convex case.

We first show that, for certain cases of {mixed} performance measures, Theorem~\ref{thm:conv_PGM_gen} can readily be used to obtain tight bounds. Then, we present additional results that were obtained \emph{numerically} using the performance estimation approach, for which we did not manage to find an analytical proof.

{\begin{proposition}
Consider the composite convex optimization problem~\eqref{eq:SSOpt}, and a feasible  point $x\in\Rd$ (i.e. $x$ is such $F(x)<\infty$) such that $s\in\partial h(x)$. The following inequalities are satisfied:
\begin{itemize}
\item[(i)] {$\normsq{x-x_\star}\leq \frac{1}{\mu^2} \normsq{\nabla f(x)+s}$,}
\item[(ii)] {$F(x)-F_\star\leq \frac{1}{2\mu} \normsq{\nabla f(x)+s} $.}
\item[(iii)] {$\normsq{x-x_\star}\leq \frac{2}{\mu} (F(x)-F(x_\star))$,}
\end{itemize}
\label{rem:easy_str_cvxty}
\end{proposition}}
\begin{proof}All the inequalities used in the proofs are either \emph{known} inequalities (see e.g.,~\cite[Section 2.5.1]{T2017peps}), or can be deduced from known inequalities using duality between smoothness and strong convexity via Fenchel conjugation (see~\cite[Theorem 2.34]{T2017peps}  for more details).

(i) By strong convexity of $F$, we have
\[ \normsq{\nabla f(x)+s-\nabla f(x_\star)-s_\star}\geq \mu^2\normsq{x-x_\star},\]
with $s_\star\in\partial h(x_\star)$ such that $\nabla f(x_\star)+s_\star=0$. Therefore
\begin{align*}
\normsq{x-x_\star}&\leq \frac{1}{\mu^2} \normsq{\nabla f(x)+s}.
\end{align*}
(ii) By strong convexity of $F$ (and feasibility of $x$), and by using Fenchel duality between smoothness and strong convexity (see~\cite[Theorem 2.34]{T2017peps}), we have
\begin{align*}
\inner{\nabla f(x_\star)+s_\star}{x}-F(x_\star)\leq &\inner{\nabla f(x)+s}{x}-F(x)\\&+\inner{\nabla f(x_\star)+s_\star-\nabla f(x)-s}{x}\\&+ \frac{1}{2\mu}\normsq{\nabla f(x)+s-\nabla f(x_\star)-s_\star}.
\end{align*}
with $s_\star\in\partial h(x_\star)$ such that $\nabla f(x_\star)+s_\star=0$, so that we obtain:
\[ F(x)-F_\star\leq \frac{1}{2\mu}\normsq{\nabla f(x)+s-\nabla f(x_\star)-s_\star}=\frac{1}{2\mu}\normsq{\nabla f(x)+s}.\]
{(iii)  Again, by strong convexity of $F$, we have:
\begin{align*}
F(x)&\geq F_\star + \inner{\nabla f(x_\star)+s_\star}{x-x_\star}+\frac{\mu}{2}\normsq{x-x_\star}
\end{align*}
with $s_\star\in\partial h(x_\star)$ such that $\nabla f(x_\star)+s_\star=0$, and we obtain the statement.}
\end{proof}

\begin{theorem}
Consider the composite convex optimization problem~\eqref{eq:SSOpt} and a feasible starting point $x_0\in\Rd$ (i.e. $x_0$ is such $F(x_0)<\infty$) such that $s_0\in\partial h(x_0)$. {The iterates of PGM satisfy the following inequalities:}
\begin{itemize}
\item[(i)]{$\normsq{x_k-x_\star}\leq \frac{\rho^{2k}(\gamma)}{\mu^2} \normsq{\nabla f(x_0)+s_0}$,}
\item[(ii)] {$F(x_k)-F(x_\star)\leq \frac{\rho^{2k}(\gamma)}{2\mu} \normsq{\nabla f(x_0)+s_0} $,}
\item[(iii)] $\normsq{x_k-x_\star}\leq \frac{2\rho^{2k}(\gamma)}{\mu} (F(x_0)-F(x_\star))$.
\end{itemize}
\label{thm:convg_GM_for_LS}
\end{theorem}
\begin{proof}
Combine results of Theorem~\ref{thm:conv_PGM_gen} with those of Proposition~\ref{rem:easy_str_cvxty}.
\end{proof}
It turns out that the global bounds provided by the previous theorem are tight for the shorter {step sizes $0\leq \gamma\leq \frac{2}{L+\mu}$ (thus including $\gamma=\frac{1}{L}$)}. This comes from the fact that 
PGM applied on the quadratic function $f_\mu(x)$ from Section~\ref{sec:quadLB} satisfies (i),(ii) and (iii) with equality. However, we observed that the guarantees of Theorem~\ref{thm:convg_GM_for_LS} are no longer tight (i.e., they are conservative) for larger step sizes. 

All known exact global convergence guarantees for (possibly mixed) combinations performance measures are summarized in Table~\ref{tab1:exactres} below. \rev{Note that for some combinations of performance measures, no analytical \emph{global and tight} convergence guarantees are known yet.}
{\renewcommand{\arraystretch}{1.5}
\begin{table}[h!]

\begin{center}{ \caption{Summary of the global convergence guarantees proposed by Theorem~\ref{thm:conv_PGM_gen} (exact, diagonal entries) and Theorem~\ref{thm:convg_GM_for_LS} (exact for $0\leq \gamma\leq \frac{2}{L+\mu}$, off-diagonal entries). 
		}\label{tab1:exactres}
        \begin{tabular}{@{}lrrr@{}}\specialrule{2pt}{1pt}{1pt}
       Initialization & $\norm{x_0-x_\star}^2$ &  $F(x_0)-F_\star\ $ & $\norm{\nabla f(x_0)+s_0}^2 $\\
       \hline
       $\norm{x_k-x_\star}^2\leq$ & $\rho^{2k} \norm{x_0-x_\star}^2$ &   ${\frac{2}{\mu}}\rho^{2k} (F(x_0)-F_\star)$ & $\frac{1}{\mu^2}\rho^{2k} \norm{\nabla f(x_0)+s_0}^2$\\
       $F(x_k)-F_\star\leq$ & \rev{(unknown)}  & $\rho^{2k} (F(x_0)-F_\star) $ & $ \frac{1}{2\mu}\rho^{2k}\norm{\nabla f(x_0)+s_0}^2$\\
       $||\tilde{\nabla} F(x_k)||^2\leq$ & \rev{(unknown)}     & \rev{(unknown)}   & $\quad  \ \rho^{2k} \norm{\nabla f(x_0)+s_0}^2$\\
       \specialrule{2pt}{1pt}{1pt}
        \end{tabular}}
    \end{center}
        \end{table}}
        
We now focus on the case of step size $\gamma=\frac1L$ and report lower bounds (corresponding to the performance of some explicitly identified functions of the form~\eqref{eq:SSOpt}) for all combinations of performance measures in Table~\ref{tab1:exactres_1L}. We conjecture those lower bounds  to be tight, as they appear to numerically match the true worst-case values obtained using the performance estimation framework~\cite{Article:Drori,taylor2015smooth,taylor2015exact,deKlerkELS2016,drori2014contributions}
(i.e., the lower bounds numerically match the computed worst cases for all tested values of $\mu,L$ and $\gamma$).
For the cases marked by a star ($\star$), the instances of \eqref{eq:SSOpt} that provided the lower bound have the following form:
\[\min_{x\geq 0} \rev{\left(\frac\mu2 \normsq{x}+cx\right)},\]
with appropriate values for $c\in \mathbb{R}$ (see \ref{app:lowerbounds} for details). For the other cases, the lower bounds are quadratics (see Section~\ref{sec:quadLB}).
 {\renewcommand{\arraystretch}{1.5}
\begin{table}[h!]
\begin{center}{ \caption{Summary of the lower bounds obtained for the case $\gamma=1/L$, conjectured to be exact.}\label{tab1:exactres_1L}
        \begin{tabular}{@{}lrrr@{}}\specialrule{2pt}{1pt}{1pt}
       Initialization & $\norm{x_0-x_\star}^2$ &  $F(x_0)-F_\star $ & $\norm{\nabla f(x_0)+s_0}^2 $\\
       \hline
       $\norm{x_k-x_\star}^2\leq$ & $\rho^{2k} \norm{x_0-x_\star}^2$ &   ${\frac{2}{\mu}}\rho^{2k} (F(x_0)-F_\star)$ & $\frac{1}{\mu^2}\rho^{2k} \norm{\nabla f(x_0)+s_0}^2$\\
       $F(x_k)-F_\star\leq$ & $\frac{\mu}{2} \frac{\norm{x_0-x_\star}^2}{\rho^{-2k}-1}$ ($\star$)& $\rho^{2k} (F(x_0)-F_\star) $ & $ \frac{1}{2\mu}\rho^{2k}\norm{\nabla f(x_0)+s_0}^2$\\
       $||\tilde{\nabla} F(x_k)||^2\leq$ & $\frac{\mu^2\norm{x_0-x_\star}^2}{(\rho^{-k}-1)^2}$  ($\star$) &   ${2\mu \frac{F(x_0)-F_\star}{\rho^{-2k}-1}}$ ($\star$) & $\rho^{2k} \norm{\nabla f(x_0)+s_0}^2$\\
       \specialrule{2pt}{1pt}{1pt}
        \end{tabular}}
    \end{center} 

        \end{table}}
        
Table~\ref{tab1:exactres_1L_mu0} below presents bounds obtained from Table~\ref{tab1:exactres_1L} when taking the limit $\mu\rightarrow 0$, i.e., for smooth convex functions that are not strongly convex. The corresponding bounds were also validated using the performance estimation framework~\cite{taylor2015smooth,taylor2015exact}. Indeed, it can be expected that tight bounds valid in the case $\mu > 0$ lead to exact bounds in the limiting case $\mu=0$ (we also point out that many non-tight bounds found in the literature do not behave properly when taking the $\mu\to 0$ limit). We observe that only trivial bounds are valid in the non-mixed cases (diagonal entries), which emphasizes the usefulness of mixed measures for that setting. Case $(\star\star)$ is proved for the projected gradient method in~\cite[Chapter 2]{drori2014contributions}. Note that several cases in Table~\ref{tab1:exactres_1L_mu0} are marked as \emph{unbounded}. Those cases correspond to settings in which guarantees can be made arbitrarily bad. As an example, consider a family of problems of the form: \[\min_{x\geq 0} cx,\] with $x\in\mathbb{R}$. By adequately tuning $c>0$ (i.e., making it arbitrarily small), one can recover the three unbounded cases from Table~\ref{tab1:exactres_1L_mu0}.
   {\renewcommand{\arraystretch}{1.5}
\begin{table}[h!]
\begin{center}{ \caption{Bounds from Table~\ref{tab1:exactres_1L} when taking the limit $\mu\rightarrow 0$.}\label{tab1:exactres_1L_mu0}
        \begin{tabular}{@{}lrrr@{}}\specialrule{2pt}{1pt}{1pt}
       Initialization & $\norm{x_0-x_\star}^2$ &  $F(x_0)-F_\star $ & $\norm{\nabla f(x_0)+s_0}^2 $\\
       \hline
       $\norm{x_k-x_\star}^2\leq$ & $\norm{x_0-x_\star}^2$ &   Unbounded & Unbounded\\
       $F(x_k)-F_\star\leq$ & ${L} \frac{\norm{x_0-x_\star}^2}{4k}$ ($\star\star$)  & $F(x_0)-F_\star$ & Unbounded\\
       $||\tilde{\nabla} F(x_k)||^2\leq$ & $L^2\frac{\norm{x_0-x_\star}^2}{k^2}$   ($\star$) &  ${L\frac{F(x_0)-F_\star}{k}}$ ($\star$) & $ \norm{\nabla f(x_0)+s_0}^2$\\
       \specialrule{2pt}{1pt}{1pt}
        \end{tabular}}
    \end{center}
    
        \end{table}}

Finally, note that the results presented in Table~\ref{tab1:exactres}, Table~\ref{tab1:exactres_1L}, and Table~\ref{tab1:exactres_1L_mu0} can be compared with the complexity results for solving linear systems as presented in the work of Nemirovski~\cite{nemirovsky1992information} (corresponding to the case of $f$ being quadratic and $h$ being identically zero). In particular,~\cite[Section 4]{nemirovsky1992information} emphasizes the fact that, for the degenerate case $\mu=0$, only the cases $(\star)$ \rev{and $(\star\star)$} can provide nontrivial worst-case convergence results.

\section{Conclusions}\label{sec:ccl}
The main contribution of this work is to close the gap between lower and upper complexity bounds for the proximal gradient method for strongly convex composite minimization problems. 

Our proof methodology allows a clear and transparent use of the assumptions of the theorems. As an example, we observed that strong convexity was only required between certain pairs of points (between consecutive iterates, or between iterates and optimum).

Tight convergence results are still open for a variety of first-order methods and different convergence measures, as for example for accelerated schemes~\cite{Book:Nesterov}, for inexact methods~\cite{schmidt2011convergence,devolder2014first} and for coordinate descent schemes~\cite{nesterov2012efficiency}. Obtaining such tight convergence results opens the door for a better use of gradient schemes as primitive operations in more involved algorithms, but also for designing optimized first-order methods (such research directions are carried out among others in~\cite{Article:Drori,kim2014optimized} for the smooth unconstrained convex case and in~\cite{lessard2014analysis} in the strongly convex case  when using a noisy gradient
), and the corresponding lower bounds~\cite{drori2016exact}.
Finally, the performance estimation framework provides guarantees that are valid for general Euclidean norms (also with primal-dual structures, see~\cite[Section 1.1]{taylor2015exact}), but it remains open to extend the methodology for handling more general norms.

\paragraph{Code} In order to ease the verification process, codes for symbolically verifying the main proofs of the paper are available from:\\ {\url{https://github.com/AdrienTaylor/ProximalGradientMethod}}\\	
All results can also be numerically verified using the Performance Estimation Toolbox~\cite{taylorperformance}, that can be downloaded from the following location:\\ {\url{https://github.com/AdrienTaylor/Performance-Estimation-Toolbox}}

{\section*{Appendices}}
\appendix
\renewcommand\thesubsection{Appendix \Alph{subsection}}
\begin{spacing}{1.0}
\subsection{- Details of the Proof of Theorem~\ref{thm:convg_GM_x}}\label{app:proofhelper}
In this section, we provide some details on the verification of the proof of Theorem~\ref{thm:convg_GM_x} (convergence in distance to optimality). The proofs of Theorem~\ref{thm:convg_GM_G} (convergence in residual gradient norm) and Theorem~\ref{thm:convg_GM_f} (convergence in function value) follow the same lines, although the results in function values are technically more involved (we advise the reader to use appropriate computer algebra software to preserve his sanity).

As the proofs for both regimes (small and large step sizes) follow the same lines, we only consider the case $0\leq \gamma\leq \frac{2}{L+\mu}$ here.

The goal is to prove that the inequality
\begin{equation}
\begin{aligned}
\left(1-\gamma \mu \right)^2 \normsq{x_k-x_\star} \geq &\normsq{x_{k+1}-x_\star}  +\gamma^2 \normsq{g_\star+s_{k+1}}\\&+\frac{\gamma(2-\gamma (L+\mu))}{L-\mu} \normsq{\mu {(x_k  -x_\star)} - g_k+  g_\star},
\end{aligned}\label{app:eq:target_expr}
\end{equation}
can be obtained by performing a weighted sum of the following inequalities:
\begin{align*}
&&&\begin{array}{ll}
f_\star\geq f_k&+\inner{g_k}{x_\star-x_k}+\frac{1}{2L}\normsq{g_k-g_\star} \\ & +\frac{\mu}{2(1-\mu/L)}\normsq{x_k-x_\star-\frac{{1}}{L}(g_k-g_\star)}
\end{array}
\quad &&:2\gamma\rho(\gamma),  \\
&&&\begin{array}{ll}
f_k\geq f_\star&+\inner{g_\star}{x_k-x_\star}+\frac{1}{2L}\normsq{g_k-g_\star} \\ &+\frac{\mu}{2(1-\mu/L)}\normsq{x_k-x_\star-\frac{{1}}{L}(g_k-g_\star)}
\end{array}
\quad &&:2\gamma\rho(\gamma),\\
&&&\begin{array}{ll}
h_\star\geq h_{k+1}+\inner{s_{k+1}}{x_\star-x_{k+1}} & \end{array}\quad &&:2\gamma,\\
&&&\begin{array}{ll}h_{k+1}\geq h_\star + \inner{s_\star}{x_{k+1}-x_\star} & \end{array} \quad &&:2\gamma.
\end{align*}
For simplicity, let us first sum the previous inequalities two by two:
\begin{align*}
&&&\begin{array}{ll}
0\geq -\inner{g_\star-g_k}{x_\star-x_k} +\frac{1}{L}\normsq{g_k-g_\star} +\frac{\mu}{1-\mu/L}\normsq{x_k-x_\star-\frac{{1}}{L}(g_k-g_\star)}
\end{array}
\quad &&:2\gamma\rho(\gamma),  \\
&&&\begin{array}{ll}
0\geq -\inner{s_\star-s_{k+1}}{x_\star-x_{k+1}} & \end{array}\quad &&:2\gamma.
\end{align*}
We proceed by showing that~\eqref{app:eq:target_expr} can obtained by reformulation of the following expression:
\begin{equation}
\begin{aligned}
0\geq &2\gamma\rho(\gamma)\left[\inner{g_k-g_\star}{x_\star-x_k} +\frac{1}{L}\normsq{g_k-g_\star} +\frac{\mu}{1-\mu/L}\normsq{x_k-x_\star-\frac{{1}}{L}(g_k-g_\star)}\right]\\
&+2\gamma  \inner{s_{k+1}-s_\star}{x_\star-x_{k+1}}.
\end{aligned}\label{app:eq:orig_expr}
\end{equation}
For doing that, we simply verify that the expression~\eqref{app:eq:target_expr} minus the expression~\eqref{app:eq:orig_expr} is identically zero; that is:
\begin{equation}
\begin{aligned}
0= &2\gamma\rho(\gamma)\left[\inner{g_k-g_\star}{x_\star-x_k} +\frac{1}{L}\normsq{g_k-g_\star} +\frac{\mu}{1-\mu/L}\normsq{x_k-x_\star-\frac{{1}}{L}(g_k-g_\star)}\right]\\
&+2\gamma  \inner{s_{k+1}-s_\star}{x_\star-x_{k+1}}\\
&-\left[-\left(1-\gamma \mu \right)^2 \normsq{x_k-x_\star}+\normsq{x_{k+1}-x_\star}  +\gamma^2 \normsq{g_\star+s_{k+1}}\right]\\
&-\frac{\gamma(2-\gamma (L+\mu))}{L-\mu} \normsq{\mu {(x_k  -x_\star)} - g_k+  g_\star},
\end{aligned}\label{app:eq:verif}
\end{equation}
with $x_{k+1}=x_k-\frac{1}{L}\left(g_k+s_{k+1}\right)$ and $s_\star=-g_\star$. 

Finally, one can simply verify the equality~\eqref{app:eq:verif} by expanding~\eqref{app:eq:target_expr} and~\eqref{app:eq:orig_expr}, which are both equal to:
\begin{equation*}
\begin{aligned}
0\geq \frac{2}{L-\mu }&\left((\gamma -\gamma ^2\mu)  \normsq{g_k} +{(\gamma ^2\mu+\gamma ^2 L-2 \gamma   )\inner{g_k}{g_\star} }\right. \\&+{( \gamma ^2 L-\gamma^2\mu) \inner{g_k}{s_{k+1}}}+{(\gamma ^2 \mu ^2+\gamma ^2  \mu  L-\gamma  L -\gamma \mu) \inner{g_k}{x_k}}\\&+{(- \gamma ^2 \mu ^2-\gamma ^2\mu L +\gamma L+\gamma \mu)\inner{g_k}{x_\star}}+{(\gamma-\gamma ^2 \mu) \normsq{g_\star} }\\&+{( \gamma ^2  L-\gamma ^2 \mu  )\inner{g_\star}{s_{k+1}}}+{(2\gamma \mu- \gamma ^2 \mu ^2-\gamma ^2  \mu  L )\inner{g_\star}{ x_k}}\\&+{( \gamma ^2 \mu ^2+\gamma ^2\mu  L-2 \gamma\mu) \inner{g_\star }{x_\star}}+{( \gamma ^2 L-\gamma ^2 \mu) \normsq{s_{k+1}}}\\&+{(\gamma  \mu-\gamma  L )  \inner{s_{k+1}}{ x_k}}+{(\gamma  L-\gamma  \mu) \inner{s_{k+1}}{ x_\star}}\\&+{(\gamma  \mu  L -\gamma ^2 \mu ^2 L) \normsq{x_k}}+{(2 \gamma ^2 \mu ^2 L-2 \gamma  \mu  L) \inner{x_k}{ x_\star}}\\&\left.+{( \gamma  \mu  L-\gamma ^2 \mu ^2 L) \normsq{x_\star}}\right).
\end{aligned}
\end{equation*}
\qed

Note that all proofs presented in the symbolic validation code rely on the exact same idea: the equivalences between pairs of expressions are verified by checking their differences being identically zero.

\subsection{- Details on Lower Bounds for Mixed Performance Measures}\label{app:lowerbounds}
In this section, we provide details for obtaining the lower bounds marked with ($\star$) in Table~\ref{tab1:exactres_1L} (i.e., those that do not come from purely quadratic functions).

First, consider two constants $0<\mu\leq L<\infty$ and the following one-dimensional quadratic minimization problem
\[\min_{x\geq 0} \left(\frac\mu2 x^2+cx\right).\]
The function $f(x)=\frac\mu2 x^2+cx$ is clearly $L$-smooth and $\mu$-strongly convex with unique optimal point $x_\star=0$ over the nonnegative reals. A corresponding composite problem can be written as $F(x)=f(x)+i_{\geq 0}(x)$ with $i_{\geq 0}(.)$ being the indicator function for the nonnegative real half-line.

Second, consider the starting point $x_0\geq 0$ and a number of iterations $N\in\mathbb{N}$. Using the proximal gradient method with step size $\frac{1}{L}$ results in the following rule for the iterates, assuming $c$ is small enough (i.e., such that $x_k\geq 0$ for all $0\leq k \leq N$):
\begin{align*}
x_{k+1}&=x_k-\frac1L \nabla f(x_k),\\
&=\left(1-\frac\mu L\right)x_k-\frac c L.
\end{align*}
Solving the recurrence equation provides us with the following rule for the iterates:
\[x_{k}=
\rev{\frac{(1-\kappa )^N (c+\mu x_0)-c}{\mu}},
\]
\rev{with $\kappa\defeq \frac{\mu}{L}$ the (inverse) condition number, }and the corresponding values:
\begin{align*}
&\nabla f(x_N)=\rev{(1-\kappa )^N (c+\mu x_0)},\\
&f(x_N)-f(x_\star)=\rev{\frac{(1-\kappa )^{2 N} (c+\mu x_0)^2}{2 \mu}.}
\end{align*}
By optimizing over $c$, we get the following extreme cases:
\begin{itemize}
	\item (quadratic optimization --- maximize $f(x_N)-f(x_\star)$ with respect to $c$) $c=\frac{\mu x_0}{(1-\kappa )^{-2N}-1}$ results in $F(x_N)-F(x_\star)=\frac{\mu}{2} \frac{\norm{x_0-x_\star}^2}{\rho^{-2N}-1}$,
	\item (linear optimization --- maximize $\lvert \nabla f(x_N)\rvert$ with respect to $c$ by \rev{enforcing equality in the constraint} $x_N\geq 0$\rev{; that is, we impose $x_N=0$}) $c=\frac{\mu x_0}{(1-\kappa )^{-N}-1}$ results in $||\tilde{\nabla} F(x_k)||^2=\frac{\mu^2\norm{x_0-x_\star}^2}{(\rho^{-N}-1)^2}$,
	\item (linear optimization --- maximize $\lvert \nabla f(x_N)\rvert$ with respect to $c$ by \rev{enforcing equality in the constraint} $x_N\geq 0$\rev{; that is, we impose $x_N=0$}) $c=\frac{\mu x_0}{(1-\kappa )^{-N}-1}$ (or equivalently $c=\frac{\sqrt{2\mu} \sqrt{F(x_0)-F_\star}}{\sqrt{(1-\kappa )^{-2N}-1}}$) results in $||\tilde{\nabla} F(x_N)||^2={2\mu \frac{F(x_0)-F_\star}{\rho^{-2N}-1}}$,
\end{itemize}
which match the entries marked \rev{($\star$)} in Table~\ref{tab1:exactres_1L}.
\end{spacing}

{\begin{spacing}{1.0}
\acknowledgement This paper
presents research results of the Belgian Network DYSCO (Dynamical Systems, Control, and Optimization), funded by the Interuniversity Attraction Poles Programme, initiated by the Belgian State, Science Policy Office, and of the Concerted Research Action (ARC) programme supported by the Federation Wallonia-Brussels (contract ARC 14/19-060). The scientific responsibility rests with its authors.

The authors also thank the anonymous referee for his very careful reading and comments, that include Remark~\ref{remark:fromref}.
\end{spacing}

\bibliographystyle{spmpsci_unsrt}      
\bibliography{bib_}{}   

\end{document}